\documentclass{amsart}
\usepackage[english]{babel} 
\usepackage{amssymb, amsthm, amsmath, pifont,enumerate} 
\usepackage[usenames]{color}
\usepackage[noadjust]{cite}
\usepackage[all]{xy}
\usepackage{amsfonts}
\usepackage{graphicx}
\usepackage{verbatim}
\usepackage{pstricks}
\usepackage[latin1]{inputenc}
\usepackage{etex}
\usepackage{mathrsfs}
\usepackage{pst-node}

\def\R{\mathbb{R}}
\def\Z{\mathbb{Z}}
\def\C{\mathbb{C}}

\def\d{\mathrm{d}}
\def\div{\mathrm{div}_\omega}
\def\hvf{\mathrm{VF}_\mathrm{hol}}
\def\avf{\mathrm{VF}_\mathrm{alg}}
\def\hvfw{\mathrm{VF}^\omega_\mathrm{hol}}
\def\avfw{\mathrm{VF}^\omega_\mathrm{alg}}
\def\chvf{\mathrm{CVF}_\mathrm{hol}}
\def\cavf{\mathrm{CVF}_\mathrm{alg}}
\def\chvfw{\mathrm{CVF}^\omega_\mathrm{hol}}
\def\cavfw{\mathrm{CVF}^\omega_\mathrm{alg}}
\def\lie{\mathrm{Lie}}
\def\F{\mathfrak{F}}
\def\G{\mathfrak{G}}
\def\L{\mathfrak{L}}
\def\O{\mathcal{O}}
\def\T{\mathrm{T}}
\def\H{\mathrm{H}}
\def\sC{\mathcal{C}}
\def\Z{\mathcal{Z}}
\def\B{\mathcal{B}}
\def\aut{\mathrm{Aut}_\mathrm{hol}}
\def\autw{\mathrm{Aut}_\mathrm{hol}^\omega}
\def\algiso{\cong_\mathrm{alg}}
\def\holiso{\cong_\mathrm{hol}}
\newtheorem{theorem}{Theorem}[section]
\newtheorem*{theorem*}{Theorem}
\newtheorem{corollary}[theorem]{Corollary}
\newtheorem*{corollary*}{Corollary}

\newtheorem*{maintheorem*}{Main Theorem}
\newtheorem{proposition}[theorem]{Proposition}
\newtheorem{lemma}[theorem]{Lemma}
\newtheorem*{lemma*}{Lemma}
\theoremstyle{definition}

\theoremstyle{definition}
\newtheorem{remark}[theorem]{Remark}
\theoremstyle{definition}
\newtheorem{definition}[theorem]{Definition}

\title[(Volume) Density Property of some complex manifolds]{(Volume) Density Property of a family of complex manifolds including the Koras-Russell Cubic}

\author{Matthias Leuenberger}
\address{Mathematisches Institut, Universit\"at Bern, Sidlerstrasse 5,
  CH-3012 Bern, Switzerland.}%
\email{matthias.leuenberger@bluewin.ch}

\date{\today}

\thanks{
  \mbox{\hspace{11pt}}{\it Key words}: Density property, holomorphic automorphisms,
  Koras-Russel cubic threefold.\\
}

\begin{document}
 
\begin{abstract}
We present modified versions of existing criteria for the density property and the volume density property of complex manifolds. We apply this methods to
show the (volume) density property for a family of manifolds given by $x^2y=a(\bar z) + xb(\bar z)$ with $\bar z =(z_0,\ldots,z_n)\in\C^{n+1}$ and volume form
$\d x/x^2\wedge \d z_0\wedge\ldots\wedge\d z_n$. The key step is showing that in certain cases transitivity of the action of (volume preserving) holomorphic
automorphisms implies the (volume) density property, and then giving sufficient conditions for the transitivity of this action. In
particular, we show that the Koras-Russell Cubic Threefold $\lbrace x^2y + x + z_0^2 + z_1^3=0\rbrace$ has the density property and the volume density property.
\end{abstract}
\maketitle
\section{Introduction}
The density property and the volume density property is a property of Stein manifolds with a huge amount of applications in complex geometry in several
variables. It was introduced by Varolin in \cite{varo-dens}. The fact that $\C^n$ has the density property for $n\geq2$ was already used by Anders\'en and
Lempert in \cite{AL} where they showed that holomorphic automorphisms can be approximated by some special family of automorphisms (called shear automorphisms).
Varolin realized that the main observation of Anders\'en and Lempert may be formalized and can be applied to more general complex manifolds and different
problems. Rosay and Forstneri\v{c} contributed also a lot to this progress in \cite{FR}. This area of complex analysis in several variables is nowadays called
Anders\'en-Lempert theory. The numerous applications of the density property are due to the Main Theorem of Anders\'en-Lempert theory which states that on
manifolds with density property any local phase flow on a Runge domain can be approximated uniformly on compacts by global automorphisms. The analogous
statement holds in the volume preserving case. For a deeper view into this topic we refer to the comprehensive texts \cite{kaku-state,kut,fo}.

\begin{definition}
 Let $X$ be a Stein manifold. If the Lie algebra $\lie(\chvf(X))$ generated by complete (= globally integrable) holomorphic vector fields $\chvf(X)$ on $X$ is
dense (in compact-open topology) in the Lie algebra of all holomorphic vector fields $\hvf(X)$ on $X$ then $X$ has the density property.

Let $X$ be a Stein manifold equipped with a holomorphic volume form $\omega$. If the Lie algebra $\lie(\chvfw(X))$ generated by complete volume
preserving (= vanishing $\omega$-divergence)  holomorphic vector fields $\chvfw(X)$ on $X$ is
dense in the Lie algebra of all volume preserving holomorphic vector fields $\hvfw(X)$ on $X$ then $X$ has the volume density property.
\end{definition}
Recall that a vector field $\nu$ is called volume preserving if the Lie derivative $L_\nu \omega$ vanishes where the Lie derivative is given by the formula
$L_\nu = \d i_\nu + i_\nu \d$ and $i_\nu$ is the interior product of a form with $\nu$.

In Section \ref{sec-crit} we present a criterion that is showing (volume) density property.
For the
definition of (semi-)compatible pairs and ($\omega$-)generating sets see Definitions \ref{def-scomp}, \ref{def-comp} and \ref{def-generate}. For a
vector field $\nu$ and a point $p\in X$ we denote by $\nu [p] \in\T_pX$ the tangent vector of $\nu$ at $p$.

\begin{theorem}\label{criterion}
 (1) Let $X$ be a Stein manifold such that the holomorphic automorphisms $\aut(X)$ act transitively on $X$. If there are compatible pairs $(\nu_i,\mu_i)$ such
that there
is a point $p\in X$ where the vectors $\mu_i[p]$ form a generating set of $\mathrm{T}_pX$ then $X$ has the density property.

(2) Let $X$ be a Stein manifold with a holomorphic volume form $\omega$ such that the volume preserving holomorphic automorphisms $\autw(X)$ act transitively
on $X$ and
$\mathrm{H}^{n-1}(X,\C)=0$ (where $n=\dim X$). If there are semi-compatible pairs $(\nu_i,\mu_i)$ of volume preserving vector fields such that there is a point
$p\in X$ where the
vectors $\nu_i[p]\wedge\mu_i[p]$ form a generating set of $\mathrm{T}_pX\wedge\mathrm{T}_pX$ then $X$ has the volume density property.
\end{theorem}

Note that this criterion  and its proof is very much inspired by the criteria in \cite{kaku-density,kaku-volume2} for the algebraic (volume) density property 
(see Definition \ref{def-adp}).
Actually, e.g.
for (1), the only
difference is that instead of requiring the algebraic automorphisms to act transitively on $X$ we require the holomorphic automorphisms to act 
transitively. 

In Section \ref{sec-trans} we investigate the transitivity of the action by (volume preserving) holomorphic automorphisms $\aut(X)$ (resp. $\autw(X)$) where
$X$ is given by $x^2y=a(\bar z) + xb(\bar z)$ with $\bar z =(z_0,\ldots,z_n)\in\C^{n+1}$ for some $n\geq 0$, $\deg_{z_0}(a)\leq 2$ and $\deg_{z_0}(b)\leq 1$. We
show that (after possibly reordering the $z_i$) the condition 
\begin{enumerate}[flushleft]
\item[(A)] There is some $k\geq0$ such that $\deg_{z_i}(a)\leq 2$ and $\deg_{z_i}(b)\leq 1$ for all $i\leq k$ and for all common zeroes $\bar q =
(q_0,\ldots,q_n)$ of $a,\frac{\partial a}{\partial
z_0},\ldots,\frac{\partial a}{\partial z_k}$, we have $b(\bar q)\neq 0$, and there is some $j\leq k$ such that $\frac{\partial a}{\partial
z_j}$ does not vanish along the curve $\lbrace z_i=q_i \text{ for all } i\neq j \rbrace \subset \C^{n+1}$.
\end{enumerate}
is a sufficient condition for $\aut(X)$ to act transitively on $X$. If additionally 
\begin{enumerate}[flushleft]
\item[(B)] There is some $k\geq0$ such that $\deg_{z_i}(a)\leq 2$ and $\deg_{z_i}(b)\leq 1$ for all $i\leq k$ and there is no $c\in\C^*$ for which the
polynomials $\frac{\partial a}{\partial{z_i}} + c\frac{\partial
b}{\partial{z_i}}$ for $i\leq k$ are all constant to zero.
\end{enumerate}
holds, then also $\autw(X)$ acts transitively (Proposition \ref{prop-trans}).

In Section \ref{sec-pos} we apply Theorem \ref{criterion} to these kind of surfaces for $n>0$. This leads to our
Main Theorem.
\begin{maintheorem*}
Let $n\geq 0$ and $a,b\in\C[z_0,\ldots,z_n]$ such that $\deg_{z_0}(a)\leq 2$, $\deg_{z_0}(b)\leq 1$ and not both of $\deg_{z_0}(a)$ and $\deg_{z_0}(b)$ are
equal to zero. Let $\bar z = (z_0,\ldots z_n)$. Then the
hypersurface $X=\lbrace x^2y=a(\bar z) + xb(\bar z)\rbrace$ has the density property provided that the holomorphic automorphisms $\aut(X)$ act
transitively on $X$. In particular $X$ has the density property if (A) holds or if $n=0$.

Moreover, if $\H^{n+1}(X,\C)=0$ and the volume preserving holomorphic automorphisms $\autw(X)$ act transitively on $X$ then $X$ has the volume density property 
for the volume form $\omega=\d x/x^2\wedge \d z_0\wedge\ldots\wedge \d z_n$. In
particular the transitivity condition holds if (A) and (B) hold or if $n=0$.
\end{maintheorem*}

The proof of the Main Theorem is finished in Section \ref{sec-zero} where the case $n=0$ is done by explicit calculations, not using the methods described in 
Section 2.

It is worth pointing out that the Main Theorem together with Corollary \ref{cor-koras} implies that the Koras-Russell Cubic Threefold $C = \lbrace x^2y + x +
z_0^2 + z_1^3=0\rbrace$ has the (volume) density
property. The Threefold $C$ is a famous example of a affine variety which is diffeomorphic to $\R^6$ but not
algebraically isomorphic to $\C^3$, e.g. see \cite{ma}. As an affine algebraic variety $C$ (in particular the algebraic automorphism group of $C$) is
well
understood, e.g. see \cite{dumopo}. For example it is know that the algebraic automorphisms does not act transitively on $C$. The
density property implies that the situation in the
holomorphic context is completely different. However, it is still unclear if $C$ is biholomorphic to $\C^3$. Related to this question there is a conjecture of
T\'oth and Varolin. The conjecture \cite{va} states that a manifold which
has the density property and which is diffeomorphic to
$\C^n$ is
automatically biholomorphic to $\C^n$. If the conjecture holds then the Main theorem would imply that $C$ is isomorphic to $\C^3$.

\section{Proof of Theorem \ref{criterion}}\label{sec-crit}
Let $X$ be a Stein manifold of dimension $n$, and let $\O_X$ be the sheaf of holomorphic functions on $X$.
\subsection{Preliminaries}
Let $\F$ be a coherent sheaf of $\O_X$-modules, and let
$s_1,\ldots,s_N \in \F(X)$ be global sections. The following lemmas are standard applications of
sheaf theory.
\begin{lemma}\label{lem-pre1}
 Let $p\in X$, and let $\mu_p\subset\O_X(X)$ be corresponding ideal. If the elements $s_i + \mu_p\F(X)$ span the vector space $\F(X)/\mu_p\F(X)$ then the
localizations $(s_i)_p$ generate the stalk $\F_p$.
\end{lemma}
\begin{proof}
 Let $\G_p$ be the $(\O_X)_p$-submodule of $\F_p$ generated by $(s_1)_p,\ldots, (s_N)_p$. The assumption implies that we have $(\G_p +
\mu_p\F_p)/\mu_p\F_p=
\F_p/\mu_p\F_p$. Let $\L_p = \F_p/\G_p$ then a short calculation using the isomorphism theorems for modules shows that we have $\L_p / \mu_p\L_p = 0$. The ring
$(\O_X)_p$ is local, so by the Nakayama Lemma we may lift a basis of $\L_p/\mu_p\L_p$ to a generating set of $\L_p$, and thus $\L_p = 0$. This yields 
$\G_p=\F_p$
which shows the claim.
\end{proof}
\begin{lemma}\label{lem-pre2}
 If the elements $(s_i)_p$ generate the stalks $\F_p$ for all points $p\in X$ then every global section $\nu\in\F(X)$ is of the form $\sum f_i s_i$ for some
global holomorphic functions $f_i\in \O_X(X)$.
\end{lemma}

\begin{proof}
 Consider the morphisms of sheaves $\varphi: \O_X^N \rightarrow \F$ given by $(f_i)\mapsto \sum f_i s_i$. By assumption $\varphi$ is surjective on the level of
stalk. Therefore we get the following short exact sequence of coherent sheaves:
$$ 0 \rightarrow \ker \varphi \rightarrow \O_X^N \rightarrow \F \rightarrow 0.$$
This yields the following long exact sequence:
$$\cdots\rightarrow \H^0(X,\O_X^N) \rightarrow \H^0(X,\F)\rightarrow \H^1(X,\ker \varphi)\rightarrow \cdots.$$
By Theorem B of Cartan we have $\H^1(X,\ker\varphi)=0$. Thus the left map is surjective, and therefore every global section $\nu \in \H^0(X,\F)=\F(X)$ is of 
the desired
form.
\end{proof}
Recall that a domain $Y\subset X$ is called Runge if all holomorphic function on $Y$ can be approximated uniformly on compacts $K\subset Y$ by global
holomorphic functions on $X$.
\begin{lemma}\label{lem-pre3}
 Let $Y\subset X$ be a domain of $X$ which is Runge and Stein. If the elements $(s_i)_p$ generate the stalks $\F_p$ for all points $p\in Y$ then every
global section $\nu\in\F(X)$ can be uniformly approximated on compacts $K\subset Y$ by global sections of the form $\sum f_i s_i$ for some global holomorphic
functions $f_i\in\O_X(X)$.
\end{lemma}
\begin{proof}
 Let $\nu\vert_Y\in\F(Y)$ be the restriction of $\nu$ to $Y$. By Lemma \ref{lem-pre2} we have $\nu\vert_Y = \sum g_i s_i\vert_Y$ for some holomorphic functions
$g_i\in\O_X(Y)$ on $Y$. Since $Y$ is a Runge domain we may approximate the functions $g_i$ by global functions $f_i\in\O_X(X)$ uniformly on compacts $K\subset
Y$. Thus the global section $\nu$ can be approximated by sections $\sum f_is_i$ uniformly on compacts $K\subset Y$.
\end{proof}

\subsection{Criterion for (volume) density property}
The following definition is due to \cite{kaku-density}, but adapted to the holomorphic case.
\begin{definition}\label{def-generate}
Let $p\in X$. A set $U\subset \T_pX$ is called generating set for $\T_pX$ if the orbit of $U$ of the induced action of the stablilizer $\aut(X)_p$ contains a
basis of $\T_pX$.

If $X$ has a volume form $\omega$ then a set $U\subset \T_pX\wedge\T_pX$ is called $\omega$-generating set for $\T_pX\wedge \T_pX$ if the orbit of $U$ of the
induced action of the stabilizer $\autw(X)_p$ contains a basis of $\T_pX\wedge\T_pX$.
\end{definition}
The next proposition a powerful criterion for the density property. It is a generalization of Theorem 2 in \cite{kaku-density} and the proof is similar.
\begin{proposition}\label{prop-crit}
 Let $X$ be a Stein manifold such that $\aut(X)$ acts transitively on $X$. Assume that there are complete vector fields $\nu_1,\ldots,\nu_N
\in \chvf(X)$ which generate a submodule that is contained in the closure of $\lie(\chvf(X))$ and assume that there is a point $p\in X$ such that the tangent
vectors
$\nu_i[p]\in \mathrm{T}_pX$ are a generating set for the tangent space $\mathrm{T}_pX$. Then $X$ has the density property.
\end{proposition}
\begin{proof}
We may assume that the vectors $\nu_i[p]$ contain a basis of $\T_pX$. Indeed, the vectors $\nu_i[p]$ are a generating set of $\T_pX$. Thus after adding some 
pull
backs of some vector fields $\nu_i$ by automorphisms in $\aut(X)_p$ we get the desired basis of $\T_pX$. Let $A\varsubsetneq X$ be the analytic subset of points
where the vectors $\nu_i[a]$ do not span the whole tangent space $\T_aX$. 

 Let $K\subset X$ be a compact set. After replacing $K$ by its $\O_X$-convex hull we may assume that $K$ is $\O_X$-convex. Let $Y$ be a neighborhood of $K$
which is Stein and Runge, and moreover such that the closure of $Y$ is compact. See e.g. the beginning of the proof of Theorem 7 in \cite{ty} for the existence
of
such a $Y$. 

After adding finitely many complete vector fields to $\nu_1,\ldots,\nu_N$ we get that $Y\cap A = \emptyset$. Indeed, since the closure of $Y$ is compact, $Y
\cap A$ is a finite union of irreducible analytic subsets. Let $A_0\subset A$ be an irreducible component of maximal dimension. Pick any $a\in A_0$ and
$\phi\in\aut(X)$ such that $\phi(a)\in Y\setminus A$. Since the vectors $\nu_i[\phi(a)]$ span the tangent space $\T_{\phi(a)}X$ the vectors $(\phi^*\nu_i)[a]$
span the tangent space $\T_aX$. Thus after adding some of the pull backs to $\nu_1,\ldots,\nu_N$ the component $A_0 \cap Y$ is replaced by
finitely many components of lower dimension. Repeating the same procedure, inductively we get after finitely steps a list of complete vector fields 
$\nu_1,\ldots,\nu_N$ such
that $A\cap Y =\emptyset$.

Let $\F$ be the coherent sheaf corresponding to the tangent bundle. The fact that the vectors $\nu_i[a]$ span $\T_aX$ for all $a\in Y$ translates to the fact
the elements $\nu_i + \mu_a\F$ span the vector space $\F/\mu_a\F$ for all $a\in Y$, where $\mu_a$ is the maximal ideal of $a$. Thus by Lemma \ref{lem-pre1} the
assumption of Lemma \ref{lem-pre3} holds. Therefore ever vector field on $X$ can be approximated uniformly on $K$ by elements of the form $\sum f_i \nu_i$ for
some regular functions $f_i \in \O_X(X)$. By assumption the submodule generated by $\nu_1,\ldots,\nu_N$ is contained in the closure of $\lie(\chvf(X))$ (note
that this property still holds after enlarging the list $\nu_1,\ldots,\nu_N$ in the procedure above). Therefore every holomorphic vector field is in the closure
of $\lie(\chvf(X))$.
\end{proof}

For the volume case the proof can be found in \cite{kakuconference}. For completeness indicate a proof here. We introduce the following isomorphisms. The 
article \cite{kaku-volume2} is a good reference for this methods. Note the reference is written
only for the algebraic
case, however all isomorphisms exist identically in the holomorphic case. Let $\omega$ be a holomorphic volume form. For $0\leq i\leq n$ let $\sC_i(X)$ be the
vector space of holomorphic $i$-forms on $X$. Moreover let $\Z_i(X)\subset\sC_i(X)$ be the vector space of closed $i$-forms, and let $\B_i(X)\subset\Z_i(X)$ be
the vector space of exact $i$-forms. Then we have the isomorphism:
\[ 
 \Phi: \hvf(X) \ \tilde\rightarrow \ \sC_{n-1}(X), \quad \nu \mapsto i_v\omega,
\]
and in the same spirit we also have the isomorphism $\Psi$ induced by:
\[
 \Psi: \hvf(X)\wedge\hvf(X) \ \tilde\rightarrow \ \sC_{n-2}(X), \quad \nu\wedge\mu \mapsto i_\nu i_\mu \omega.
\]
By definition $\hvfw(X)$ consists of those vector fields $\nu$ such that $L_\nu\omega=\d i_\nu\omega = 0$ holds and thus the isomorphism $\Phi$ restricts to an
isomorphisms
\[
 \Theta = \Phi\vert_{\hvfw(X)}: \hvfw(X) \ \tilde\rightarrow \ \Z_{n-1}(X).
\]
Moreover, we consider the outer differential
\[
 D: \sC_{n-2}(X)\twoheadrightarrow \B_{n-1}(X)
\]
of $(n-2)$-forms.

\begin{lemma}\label{lem-3.1}
 Let $\alpha,\beta\in\hvfw(X)$. Then $i_{[\alpha,\beta]}\omega = \d i_\alpha i_\beta \omega$.
\end{lemma}
\begin{proof}
 Proposition 3.1 in \cite{kaku-volume2}.
\end{proof}

The next proposition is in some sense the version Proposition \ref{prop-crit} in the volume preserving case. It is an adaption of Theorem 1 in
\cite{kaku-volume2}.
\begin{proposition}\label{prop-critw}
 Let $X$ be a Stein manifold such that $\autw(X)$ acts transitively on $X$. Assume that every class of $\H^{n-1}(X,\C)=\Z_{n-1}(X)/\B_{n-1}(X)$ contains
an element of $\Theta(\lie(\chvfw(X)))$. Moreover, assume that
there are complete vector fields $\nu_1,\ldots,\nu_N \in \chvfw(X)$ and $\mu_1,\ldots,\mu_N\in\chvfw(X)$ such that the submodule of
$\hvf(X)\wedge\hvf(X)$ generated by the elements $\nu_j\wedge\mu_j$ is contained in the closure of $\lie(\chvfw(X))\wedge\lie(\chvfw(X))$, and assume that there
is a point
$p\in X$ such that the vectors $\nu_j[p]\wedge\mu_j[p]$ are a $\omega$-generating set for the vector space $\mathrm{T}_pX\wedge\mathrm{T}_pX$. Then $X$ has the
volume density property.
\end{proposition}
\begin{proof}
 Let $K\subset X$ be a compact set. By the identical arguments as in the proof of Proposition \ref{prop-crit} we see that every element of
$\hvf(X)\wedge\hvf(X)$ may be uniformly approximated on $K$ by elements contained in $\lie(\chvfw(X))\wedge\lie(\chvfw(X))$.

Let $\zeta\in\hvfw(X)$. By the first assumption we may, after subtracting an element of $\Theta(\lie(\chvfw(X)))$, assume that $\Theta(\zeta)\in\B_{n-1}(X)$.
Thus $\Theta(\zeta) = D(\Psi(\gamma))$ for some $\gamma \in \hvf(X)\wedge\hvf(X)$. Let us approximate $\gamma$ uniformly on $K$ by elements of the form $\sum
\alpha_i\wedge \beta_i \in \lie(\chvfw(X))\wedge\lie(\chvfw(X))$ with $\alpha_i,\beta_i \in \lie(\chvfw(X))$. We use Lemma \ref{lem-3.1} to see that
\[
 D (\Psi(\sum\alpha_i\wedge \beta_i)) = \sum \d i_{\alpha_i} i_{\beta_i} \omega = \sum i_{[\alpha_i,\beta_i]} \omega\in \Theta(\lie(\chvfw(X))).
\]
Since $\Theta(\zeta)=i_\zeta\omega$ can be approximated uniformly on $K$ by elements of the form $\sum i_{[\alpha_i,\beta_i]} \omega = \Theta(\sum
[\alpha_i,\beta_i])$ the vector field $\zeta$ can be approximated uniformly on $K$ by elements of the form $\sum [\alpha_i,\beta_i] \in \lie(\chvfw(X))$ which
proves the proposition. \end{proof}

\subsection{Semi-compatible and compatible pairs} This section provides a way to find the submodules which are required in Propositions \ref{prop-crit} and
\ref{prop-critw}. The parts \ref{def-scomp} - \ref{lem-comp} are e.g. from \cite{kaku-density} and \cite{kaku-volume2} and are here adapted to the holomorphic
case. For a vector field $\nu$ and a regular function $f$ we denote by $\nu(f)$ the regular function which is obtained by applying $\nu$ as a derivation and
$\ker \nu$ is the kernel of this linear map.
\begin{definition}\label{def-scomp}
 A semi-compatible pair is a pair $(\nu,\mu)$ of complete vector fields such that the closure of the linear span of the product of the kernels
$\ker\nu \cdot\ker\mu$
contains a non-trivial ideal $I \subset\O_X(X)$. We call $I$ the ideal of $(\nu,\mu)$.
\end{definition}
\begin{lemma}\label{lem-scomp}
 Let $(\nu,\mu)$ be a semi-compatible pair of volume preserving vector fields and $I$ be its ideal. Then the submodule of $\hvf(X)\wedge\hvf(X)$ given by
$I\cdot(\nu\wedge\mu)$ is contained in the closure of $\lie(\chvfw(X))\wedge\lie(\chvfw(X))$
\end{lemma}
\begin{proof}
 Let $\tau=(\sum f_ig_i)\cdot(\nu\wedge\mu)$ with $f_i\in\ker\nu$ and $g_i\in\ker\mu$ be an arbitrary element of $\mathrm{span}\lbrace\ker\nu
\cdot\ker\mu\rbrace\cdot(\nu\wedge\mu)$.
Since $f_i\in\ker\nu$ we have $f_i\nu\in\chvfw(X)$ for all $i$ and similarly $g_i\nu\in\chvfw(X)$ for all $i$. Thus $\tau=\sum f_i\nu\wedge
g_i\mu\in\lie(\chvfw(X))\wedge\lie(\chvfw(X))$. Therefore the closure of $\mathrm{span}\lbrace\ker\nu \cdot\ker\mu\rbrace\cdot(\nu\wedge\mu)$ is contained in
the closure of
$\lie(\chvfw(X))\wedge\lie(\chvfw(X))$, and the claim follows.
\end{proof}

\begin{definition}\label{def-comp}
 A semi-compatible pair $(\nu,\mu)$ is called a compatible pair if there is an holomorphic function $h\in\O_X(X)$ with $\nu(h)\in\ker\nu\setminus 0$
and $h\in\ker\mu$. Note that this condition in particular implies that $h\nu$ is a complete vector field.
\end{definition}
\begin{lemma}\label{lem-comp}
 Let $(\nu,\mu)$ be a compatible pair, and let $I$ be its ideal and $h$ its function. Then the submodule of $\hvf(X)$ given by $I\cdot \nu(h)\cdot \mu$ is
contained in the closure of $\lie(\chvf(X))$.
\end{lemma}
\begin{proof}
Let $f\in\ker\nu$ and $g\in\ker\mu$, then $f\nu,fh\nu,g\mu,gh\nu\in\chvf(X)$. A standard calculation shows
$$[f\nu,gh\mu]-[fh\nu,g\mu]=fg\nu(h)\mu\in\lie(\chvf(X)). $$ Thus an arbitrary element $\sum (f_i g_i)\nu(h)\mu\in \mathrm{span}\lbrace\ker\nu
\cdot\ker\mu\rbrace\cdot\nu(h)\cdot\mu$
with $f_i\in\ker\nu$ and $g_i\in\ker\mu$ is contained in $\lie(\chvf(X))$ which concludes the proof. \end{proof}
 \begin{proof}[Proof of Theorem \ref{criterion}]
  \textit{(1)} Let  $I_i$ be the ideals and $h_i$ the functions of the pairs $(\nu_i,\mu_i)$ and pick any non-trivial $f_i\in I_i\cdot\nu_i(h_i)$ for every $i$.
Since the set of points $p\in X$ where the vector fields $\mu_i[p]$ are a generating set is open and non-empty there is some $q\in X$ where the vector fields
$f_i(q)\mu_i[q]$ are a generating set for $\T_qX$. By Lemma \ref{lem-comp} the module generated by the vector fields $f_i\mu_i$ is contained in the closure of
$\lie(\chvf(X))$ and thus by Proposition \ref{prop-crit} the manifold $X$ has the density property.

\textit{(2)} Process as in \textit{(1)}: Let $I_i$ be the ideals of the pairs $(\nu_i,\mu_i)$ and pick any non-trivial $f_i\in I_i$ for every $i$. Since the set
of points $p\in X$ where the elements $\nu_i[p]\wedge\mu_i[p]$ are an $\omega$-generating set is open and non-empty there is a $q\in X$ where the vector fields
$f_i(q)\cdot(\nu_i[q]\wedge\mu_i[q])$ are an $\omega$-generating set for $\T_qX\wedge\T_qX$. By Lemma \ref{lem-scomp} the module generated by the elements
$f_i\cdot(\nu_i\wedge\mu_i)$ is contained in the closure of
$\lie(\chvfw(X))\wedge\lie(\chvfw(X))$. Thus by Proposition \ref{prop-critw} the manifold $X$ has the volume density property (the first condition of
Proposition \ref{prop-critw} on the cohomology group is trivially fulfilled since $\H^{n-1}(X,\C)=0$ by assumption).
 \end{proof}

We conclude this section by two remarks. These two remarks are just for general information and are not used later in this article.
\begin{remark}
 Clearly Theorem \ref{criterion}(2) still holds if we have that every class of
$\H^{n-1}(X,\C)$ contains
an element of $\Theta(\lie(\chvfw(X)))$ as in Proposition \ref{prop-critw} instead of $\H^{n-1}(X,\C)=0$. Also, note that this condition is equivalent to the 
condition that every class of
$\H^{n-1}(X,\C)$ contains
an element of $\Theta(\chvfw(X))$. Indeed, by Lemma \ref{lem-3.1} all Lie brackets represent the trivial class of $\H^{n-1}(X,\C)$.
\end{remark}

\begin{remark}
 There is another class of compatible pairs. Sometimes a semi-compatible pair $(\nu,\mu)$ is also called compatible if there exists a function $h\in\O_X(X)$
with $\nu(h)\in\ker\nu\setminus 0$ and $\mu(h)\in\ker\mu\setminus 0$. For this version the identity
$[f\nu,gh\mu]-[fh\nu,g\mu]=fg(\nu(h)\mu-\mu(h)\nu)$ implies that there would be a version of Theorem \ref{criterion} where we allow compatible pairs of this
kind such that the vectors $\nu(h)\mu - \mu(h)\nu$ take part in constructing the generating sets.
\end{remark}
\section{Transitivity of the $\aut(X)$- and $\autw(X)$-action}\label{sec-trans}
Let $n\geq k\geq 0$ and $a,b\in\C[z_0,\ldots,z_n]$ such that $\deg_{z_i}(a)\leq2$ and $\deg_{z_i}(b)\leq 1$ for all $i\leq k$. Let
$\bar z =
(z_0,\ldots z_n)$ and $X=\lbrace x^2y=a(\bar z) + xb(\bar z)\rbrace$ with the holomorphic volume form $\omega = \d x/x^2\wedge\d
z_0\wedge\ldots\wedge \d z_n$\footnote{It is a general fact that on every hypersurface $\lbrace P = 0\rbrace\subset\C^N$ there is a natural volume form given 
by $\omega = \left(\frac{\partial P}{\partial x_i}\right)^{-1} \d x_1\wedge\ldots\wedge\widehat{\d x_i}\wedge\ldots\wedge \d x_N$.}. Consider the following 
vector fields on $X$:
\[ v_x^i = \left(\frac{\partial a}{\partial z_i}+x\frac{\partial b}{\partial z_i}\right)\frac{\partial}{\partial y} + x^2\frac{\partial}{\partial z_i}
\quad \text{and} \quad
 v_y^j = \left(\frac{\partial a}{\partial z_j}+x\frac{\partial b}{\partial z_j}\right)\frac{\partial}{\partial x} + \left( 2xy - b(\bar z)\right)
\frac{\partial}{\partial z_j}
\]
for $0\leq i \leq n$ and $0\leq j \leq k$ and moreover let
\[v_z = a(\bar z)x\frac{\partial}{\partial x}-\left(2a(\bar z)y -xyb(\bar z) + b(\bar z)^2\right)\frac{\partial}{\partial y}.
\]

\begin{lemma}\label{lem-complete}
$f(\bar z)v_z\in\chvf(X)$, $f(x,z_0,..\hat{z_i}..,z_n)v_x^i \in \chvfw(X)$ and \linebreak $f(y,z_0,..\hat{z_j}..,z_n)v_y^j \in \chvfw(X)$ for
$f:\C^{n+1}\rightarrow \C$,
$0\leq i \leq n$ and $0\leq
j \leq k$.
\end{lemma}
\begin{proof}
 It is easy to see that the vector fields $v_x^i$ are locally nilpotent, so they are in particular complete. The coefficients of $v_y^j$ are linear in $x$ and
$z_j$ for all $0\leq j\leq k$ so the flow equation is a linear differential equation, and thus has a global solution. The vector field $v_z$ is
complete since we may first solve the linear and uncoupled differential equation for $x$. Then the differential equation for $y$ becomes linear and uncoupled 
as well, and thus we have a global solution. It is left to show that the vector fields
$v_x^i$ and $v_y^j$ are volume preserving. A standard calculation shows that
\begin{eqnarray*}
 i_{v_x^i}\omega & = &(-1)^{i+1}\d x\wedge \d z_0 \wedge ..\widehat{\d z_i}..\wedge \d z_n, \\
 i_{v_y^j}\omega &= &\frac{1}{x^2}\left(\frac{\partial a}{\partial z_j}+x\frac{\partial b}{\partial z_j}\right)\d z_0 \wedge \ldots \wedge \d z_n + \\
& & \frac{(-1)^{j+1}}{x^2} ( 2xy -b(\bar z) ) \d x\wedge \d z_0 \wedge..\widehat{\d z_j}..\wedge \d z_n \\
& = & (-1)^j \ \d \left( \frac{a(\bar z) + xb(\bar z)}{x^2} \right) \wedge \d z_0 \wedge ..\widehat{\d z_j}..\wedge \d z_n \\
& = & (-1)^j  \d y \wedge \d z_0 \wedge ..\widehat{\d z_j}..\wedge \d z_n.
\end{eqnarray*}
Thus $L_{v_x^i}\omega = \d i_{v_x^i}\omega = 0$ and $L_{v_y^j}\omega = \d i_{v_y^j}\omega = 0$ which shows that they are volume preserving.

Multiplying with a kernel element doesn't affect these properties.
\end{proof}

\begin{lemma}
 The group $\aut(X)$ acts transitively on $X\setminus \lbrace x=0\rbrace$. Moreover $\autw(X)$ acts transitively on $X\setminus \lbrace x=0\rbrace$ provided
that
condition (B) from the introduction holds.
\end{lemma}
\begin{proof}
 Use the flow of the vector fields $v_x^i$ to get a transitive action on the fibers $\lbrace x=c\neq0\rbrace$ and the fields $v_z$ to connect the fibers. Thus
we see
that $\aut(X)$ acts transitively on $X\setminus \lbrace x=0\rbrace$. If condition (B) holds then we may also use the vector fields $v_y^j$ to connect the
fibers. Indeed, for every $c\neq 0$ there is a point $p\in \lbrace x=c\rbrace$ and a vector field $v_y^j$ such that $v_y^j[p]$ points outwards of the fiber.
Thus in this case $\autw(X)$ acts transitively on $X\setminus \lbrace x=0\rbrace$.
\end{proof}
\begin{lemma} If (A) from the introduction holds then for every point $p\in X\cap \lbrace x=0\rbrace$ there is some $\phi\in\autw(X)$ such that
$\phi(p)\notin\lbrace x=0\rbrace$.
\end{lemma}
\begin{proof}
By (A) we have that for any point $p=(0,y_0,\bar q)\in
\lbrace
x=0\rbrace\cap X$
at least one of the polynomials $b,\frac{\partial a}{\partial
z_0},\ldots,\frac{\partial a}{\partial z_k}$ does not vanish at $\bar q$. For a non-vanishing $\frac{\partial a}{\partial z_j}$ the vector field $v_y^j$ points
outwards from $\lbrace x=0 \rbrace$ at the point $p$. Thus the flow of $v_y^j$ moves $p$ away from $\lbrace x = 0\rbrace$.
If all polynomials
$\frac{\partial a}{\partial
z_0},\ldots,\frac{\partial a}{\partial z_k}$ vanish at $\bar q=(q_0,\ldots,q_n)$ then $b$ is non-vanishing at $\bar q$ and moreover there is a $j\leq k$ such
that $\frac{\partial a}{\partial
z_j}$ does not vanish along the curve $\lbrace z_i=q_i \text{ for all } i\neq j \rbrace \subset \C^{n+1}$. This means that $v_y^j$ is non-vanishing at $p$.
Assume the orbit of $p$ by the flow of $v_y^j$ is contained in $\lbrace x=0\rbrace$ then the set $\lbrace x=0,y=y_0,z_i=q_i \text{ for all } i\neq j \rbrace
\subset \C^{n+3}$ would contained in $X$ and tangent to $v_y^j$, which is not the case.
\end{proof}
These two lemmas combined give the following proposition using the conditions from the introduction.
\begin{proposition}\label{prop-trans} Assume that (A) holds. Then $\aut(X)$ acts transitively on $X$. Assume that additionally (B) holds.
Then $\autw(X)$
acts transitively on $X$. 
\end{proposition} 

\begin{corollary}\label{cor-koras}
 The volume preserving automorphisms act transitively on the Koras-Russell cubic $C=\lbrace x^2y + x + z_0^2 + z_1^3=0\rbrace$.
\end{corollary}
\begin{proof}
 We have $\frac{\partial a}{\partial z_0} = -2z_0$ and $b(z_0,z_1)=-1$. Thus it is easy to see that (A) and (B) hold, and thus the vector fields $v_x^0,v_x^1$
and $v_y^0$ induce a transitive action on $C$.
\end{proof}

\begin{remark} 
 The transitivity of the action by automorphisms a priori need not to be achieved by the vector fields $v_x^i,v_y^i,v_z$ only (which is equivalent to condition
(A) from the introduction). There could be further automorphisms. For example depending on $a$ and $b$ there could be automorphisms of the form $(x,y,\bar
z)\mapsto(x,\gamma y,\lambda(\bar z))$ where $\lambda$ is an automorphisms of $\C^{n+1}$ with the property that $a(\lambda(\bar z))+xb(\lambda(\bar
z))=\gamma(a(\bar z) + xb(\bar z))$ for some $\gamma\in\C^*$. A similar statement holds for transitivity by volume preserving automorphisms.
\end{remark}

\section{The Main Theorem for $n>0$}\label{sec-pos}
Let $n>0$, $n\geq k\geq 0$ and $a,b\in\C[z_0,\ldots,z_n]$ such that $\deg_{z_i}(a)\leq2$ and $\deg_{z_i}(b)\leq 1$ for all $i\leq k$. Moreover, assume that
not both of
$\deg_{z_0}(a)$ and $\deg_{z_0}(b)$ are equal to
zero. Let
$\bar z=(z_0,\ldots z_n)$ and $X=\lbrace x^2y=a(\bar z) + xb(\bar z)\rbrace$ with the holomorphic volume form $\omega = \d x/x^2\wedge\d
z_0\wedge\ldots\wedge \d z_n$.

Consider, again, the following vector fields on $X$:
\[ v_x^i = \left(\frac{\partial a}{\partial z_i}+x\frac{\partial b}{\partial z_i}\right)\frac{\partial}{\partial y} + x^2\frac{\partial}{\partial z_i}
\quad \text{and} \quad
 v_y^j = \left(\frac{\partial a}{\partial z_j}+x\frac{\partial b}{\partial z_j}\right)\frac{\partial}{\partial x} + \left( 2xy - b(\bar z)\right)
\frac{\partial}{\partial z_j}
\]
for $0\leq i \leq n$ and $0\leq j \leq k$.

\begin{lemma}\label{lem-pairs}
 Let $0\leq i\leq n$ and $0\leq j \leq k$. Then $(v_x^i,v_y^j)$ are compatible pairs for $i\neq j$. 
\end{lemma}
First we show that $(v_x^i,v_y^j)$ are semi-compatible pairs. Indeed the kernel of $v_x^i$ is contained in the functions depending on $x,z_0,..\hat z_i ..,z_n$ 
and
the kernel of $v_y^j$ is contained in the functions depending on $y,z_0,..\hat z_j ..,z_n$ thus the closure of $\mathrm{span}\lbrace\ker v_x^i\cdot \ker
v_y^j\rbrace$ is equal to $\O_X(X)$ and
in particular contains an ideal.

For $(v_x^i,v_y^j)$ being a compatible pair we need a function $h\in\ker v_y^j$ such that $v_x^i(h)\in\ker v_x^i\setminus 0$. The function $h=z_i$ does the
job.
\begin{lemma}\label{lem-genset}
 For a generic point $p\in X$ the vector $v_y^0[p]$ is a generating set for $\mathrm{T}_pX$ and the set 
$v_x^n[p]\wedge v_y^0[p]$ is a $\omega$-generating set for $\mathrm{T}_pX\wedge \mathrm{T}_pX$. Note that the first statement also true for $n=0$.
\end{lemma}
\begin{proof}
Let $p=(x_0, y_0 ,\bar q)$ where $\bar q=(q_0,\ldots,q_n)$ be such that $x_0\neq 0$ and $\frac{\partial a(\bar q)}{\partial z_0}+x_0\frac{\partial
b(\bar q)}{\partial z_0}\neq0$. 

For a complete vector field $\nu\in\chvf(X)$ and kernel element $f\in\ker\nu$ with $f(p)=0$ we get an induced action (by the time one flow map of $f\nu$) on
$\T_pX$ given by $v\mapsto v+v(f)\nu[p]$. Let $\nu_i = v_x^i$ and $f_i = x - x_0$ for $0\leq i\leq n$. Thus the orbit of $v_y^0[p]$ under the
$\autw(X)_p$-action contains the vectors $v_y^0[p] + \left(\frac{\partial a(\bar q)}{\partial z_0}+x_0\frac{\partial
b(\bar q)}{\partial z_0}\right)v_x^i[p]$. Therefore the orbit contains a basis for $\T_pX$.

Since $f\in\ker v_x^n$ we have that $v_x^n[p]\mapsto v_x^n[p]$ under the actions given by $(x-x_0)v_x^i$. So in particular, similarly we have that 
the orbit of
$v_x^n[p]\wedge v_y^0[p]$ contains a basis for $v_x^n[p]\wedge \T_pX$. Considering now the actions given by the vector fields $(z_n - q_n)v_x^i$ for $i\leq
n-1$. We get the actions $v_x^n[p]\mapsto v_x^n[p] + x_0^2v_x^i[p]$, and thus we see that the orbit of the $\autw(X)_p$-action contains a basis for 
$v_x^i[p]\wedge
\T_pX$ for all $i\leq n$. Together they build then together a basis for $\T_pX\wedge \T_pX$.
\end{proof}
\begin{proof}[Proof of the Main Theorem for $n>0$]
By Lemma \ref{lem-pairs} there exists a point $p\in X$ and compatible pairs
$(\nu_i,\mu_i)$ such that the vectors $\mu_i[p]$ are a generating set for $\T_pX$ and the elements $\nu_i[p]\wedge\mu_i[p]$ are a $\omega$-generating set for
$\T_pX\wedge\T_pX$ thus by Theorem \ref{criterion} the claim is proven. Proposition \ref{prop-trans} proves the ``in particular'' part.
\end{proof}

\begin{remark}
We never used that $a$ and $b$ are polynomials. In fact, the Main Theorem also holds if $a$ and $b$ are polynomial in $z_0$ and analytic in $z_1,\ldots,z_n$. 
\end{remark}

\begin{remark}
 The condition that $\H^{n+1}(X,\C)=0$ in the Main Theorem could be omitted in the case when every class of $\H^{n+1}(X,\C)$ contains
an element of $\Theta(\lie(\chvfw(X))$ as in Proposition \ref{prop-critw}. However this cannot be achieved by the complete vector fields $f\cdot v_x^i$ and
$g\cdot v_y^j$ since they are all mapped to the zero class by $\Theta$ (see the calculation in the proof of Lemma \ref{lem-complete}). So the existence of
other complete volume preserving vector fields would be required.
\end{remark}

\begin{remark} If $a(\bar z)$ is reduced (i.e. has connected fibers) and $\lbrace a(\bar z) = 0\rbrace\subset\C^{n+1}$ is smooth then the condition 
$\H^{n+1}(X,\C)$ is actually equivalent
to the condition that $\tilde\H^{n-1}(\lbrace a(\bar z) = 0\rbrace,\C) = 0$. Indeed, if $Y=\lbrace uv = a(\bar w)\rbrace$ then the map $X\rightarrow Y$ given 
by $(x,y,\bar z)\mapsto (x, xy -b(\bar z),\bar z)$ is an
affine modification in the sense of \cite{za}. Theorem 3.1 of \cite{za} shows that $\H^*(X,\C)=\H^*(Y,\C)$. Moreover, Proposition 4.1 of \cite{za} states
that we have
$$\tilde\H^*(Y,\C) = \tilde\H^{*-2}\Big(\lbrace a(\bar z)=0\rbrace,\C\Big)$$ which proofs the statement of the remark.
\end{remark}

\section{The Main Theorem for $n=0$}\label{sec-zero}
Let $a,b\in\C$, and let $X_{a,b}=\lbrace x^2y = z^2 - b + ax \rbrace$. Note that $X_{a,b}$ is smooth if and only if $a$ and $b$ are not both equal to zero. 
Also, the conditions (A) and (B) from the introduction hold automatically if $X_{a,b}$ is smooth. Therefore the volume preserving automorphisms act 
transitively on
$X_{a,b}$. The following proposition is mostly taken from \cite{po}, and among others it shows that $X_{a,b}$ is algebraically isomorphic to $X_{1,0}$, 
$X_{0,1}$ or
$X_{1,1}$. Moreover it shows that $X_{0,1}$ and $X_{1,1}$ are biholomorphic.

\begin{proposition}\label{standard}
\begin{enumerate}[(a)]
\item  Let $a\in\C$ and $b\in\C^*$ then
\begin{enumerate}[(i)]
 \item $X_{a,1}\algiso X_{a,b}$ and $X_{b,a}\algiso X_{1,a}$,
 \item $X_{0,1}\holiso X_{a,1}$,
 \end{enumerate}
where $\algiso$ means isomorphic as algebraic surfaces and $\holiso$ isomorphic as complex manifolds.
\item  Let $X = \lbrace x^2y=p(x,z)\rbrace$ be a smooth hypersurface with $p\in\C[x,z]$ and $\deg_zp(0,z)\leq 2$ then
\begin{enumerate}[(i)]
 \item if $\deg_zp(0,z)= 0$ then $X\algiso\C^*\times \C$,
 \item if $\deg_zp(0,z)= 1$ then $X\algiso\C^2$,
\item if $p(0,z)$ has a zero with multiplicity 2 then $X\algiso X_{1,0}$,
\item if $p(0,z)$ has two different zeroes then there is a unique $a\in\lbrace 0,1\rbrace$ such that $X\algiso X_{a,1}$.
\end{enumerate}
\end{enumerate}
\end{proposition}
\begin{proof}
Theorem 9 from \cite{po} gives the isomorphisms in (a)(i). The biholomorphic map in (a)(ii) is given by
$$ (x,y,z)\mapsto \left(x,e^{-ax}y +\frac{e^{-ax}+ax -1}{x^2},e^{-\frac{a}{2}x}z\right).$$
Theorem 5 from \cite{po} states that $X$ is algebraically isomorphic to $x^2y = s(z) + xt(z)$
for some $s,t\in\C[z]$ with $\deg t \leq \deg s -2$ Following the given algorithm we see that $s(z)=p(0,z)$ which is then, after a linear change in the
$z$-coordinate, given by (i)~$1$, (ii)~$z$, (iii)~$z^2$ or (iv)~$z^2-1$. The isomorphisms in (b) are then easily found and the uniqueness in (b)(iv) follows 
from
Theorem 9 from \cite{po}.
\end{proof}

Despite of Proposition \ref{standard}(a) we will start working on $X_{a,b}$ for general $a,b\in\C$. It turns out that this is more convenient for most 
arguments.

\begin{remark}\label{rem-standard}
 Every function $f\in\C[X_{a,b}]$ can be written uniquely as
\[
 f(x,y,z)=\sum_{i=1}^\infty x^ia_i(z) + \sum_{i=1}^\infty xy^ib_i(z) + \sum_{i=1}^\infty y^ic_i(z) + d(z),
\]
indeed replace every $x^2y$ by $z^2 - b + ax$. Alternatively $f$ can be written uniquely as
\[ 
 f(x,y,z)= f_1(x,y) + zf_2(x,y),
\]
indeed replace every $z^2$ by $x^2y + b -ax$.
\end{remark}

We will use a tool in order to proof the (volume) density property, namely the algebraic (volume) density property. Note that the algebraic (volume)
density property implies the (volume) density property, see \cite{kaku-volume}.

\begin{definition}\label{def-adp}
 Let $X$ be an affine algebraic manifold. If the Lie algebra \linebreak $\lie(\cavf(X))$ generated by complete algebraic vector fields
$\cavf(X)$ on $X$ is equal to the Lie algebra of all algebraic vector fields $\avf(X)$ on $X$ then $X$ has the algebraic density property.

Let $X$ be an affine algebraic manifold equipped with an algebraic volume form $\omega$. If the Lie algebra $\lie(\cavfw(X))$ generated by complete volume
preserving algebraic vector fields $\cavfw(X)$ on $X$ is equal to the Lie algebra of all volume preserving algebraic vector fields $\hvfw(X)$ on $X$ then $X$
has the algebraic volume density property.
\end{definition}

\subsection{The volume density property}
Let $a,b\in\C$ such that not both equal to zero. For proving the volume density property for $X_{a,b}= \lbrace x^2y=z^2 - b + ax\rbrace$ with respect to
$\omega = \d x/x^2\wedge \d z$ we will need the following two vector
fields:
\[
 v_x=2z\frac{\partial}{\partial y} + x^2\frac{\partial}{\partial z}, \quad
 v_y=2z\frac{\partial}{\partial x} + (2xy-a)\frac{\partial}{\partial z}.
\]
Lemma \ref{lem-complete} translates into: 
\begin{lemma}\label{complete}
 $x^kv_x\in\cavfw(X_{a,b})$ and $y^kv_y\in\cavfw(X_{a,b})$ for $k\geq0$.
\end{lemma}

\begin{lemma}\label{correspondance}
 Let $v\in\avfw(X_{a,b})$. Then the 1-form $i_v \omega$ is exact and $i_v \omega = \d f$ defines a bijection between algebraic volume preserving vector fields
and
algebraic functions modulo constants. 

The functions corresponding to $x^kv_x$ and $y^kv_y$ are given by the equations
\[(k+1)i_{(x^kv_x)}\omega=-\d x^{k+1} \text{ and } (k+1)i_{(y^kv_k)}\omega =\d y^{k+1}.\]
\end{lemma}
\begin{proof}
 The correspondence is given by the isomorphism $\Theta$ the map $D$ in Section \ref{sec-crit} using the fact that $\H^1(X_{a,b},\C) = 0$. The same
correspondence was
also used in \cite{le}. The triviality
of $\H^1(X_{a,b},\C)$ follows from the fact that $X_{a,b}$ is an affine modification of $\C^2$ along the divisor $2\cdot\lbrace x=0\rbrace$ with center at the
ideal
$(x^2,z^2 - b + ax)$ using the notation from \cite{za}. If $b\neq 0$ then Proposition 3.1 from \cite{za} shows that $X_{a,b}$ is simply connected (since $\C^2$
is simply connected). If $b = 0$ then Theorem 3.1 from \cite{za} shows in a similar way that $\H^1(X_{a,0},\C) = \H^1(\C^2,\C)$, and thus is trivial.
For the two identities we make the
calculations:
$$ (k+1)i_{(x^kv_x)}\omega= (k+1)x^k i_{v_x}\omega=-(k+1)x^k\d x = -\d x^{k+1},$$
$$ (k+1)i_{(y^kv_y)}\omega= (k+1)y^k i_{v_y}\omega= (k+1)y^k \left( \frac{2z}{x^2}\d z - \frac{2xy -a}{x^2}\d x\right) = $$ $$= (k+1)y^k\d
\frac{z^2 - b + ax}{x^2}=(k+1)y^k\d y = \d y^{k+1}.$$
\end{proof}
Recall that for a vector
field $\nu$ and a regular function $f$ we denote by $\nu(f)$ the regular function which is obtained by applying $\nu$ as a derivation.
\begin{lemma}\label{bracket}
Let $v_1,v_2\in\avfw(X_{a,b})$ and $i_{v_1} \omega =\d f$ then $i_{[v_2,v_1]}\omega=\d v_2(f)$. In particular, if $f$ corresponds to a vector field in
$\lie(\cavfw(X_{a,b}))$
then
$x^kv_x(f)$ and $y^kv_y(f)$ correspond to a vector field in $\lie(\cavfw(X_{a,b}))$.  
\end{lemma}
\begin{proof}
 The identity $i_{[v_2,v_1]}\omega=\d v_2(f)$ is shown in Lemma 3.2 in \cite{le}.
\end{proof}

\begin{lemma} Let $i,j,k\geq 0$. Then
\begin{eqnarray}\label{4} v_x(y^{j+1}) &=& 2(j+1)y^j z,\\
\label{5} v_x(y^{j+1}z^{k+1})&=&y^jz^k(2(j+1)z^2+(k+1)(z^2 - b + ax)),\\
\label{1} y^jv_y(z^{k+1})&=&(k+1)y^jz^k(2xy-a),\\
\label{2} v_y(x^{i+1}) &=& 2(i+1)x^i z,\\
\label{3} v_y(x^{i+1}z^{k+1})&=& x^iz^k\left( 2(i+1)z^2 + (k+1)(2z^2 - 2b +ax)\right). \end{eqnarray}
\end{lemma}
\begin{proof}
 The lemma is proven by straight forward calculations.
\end{proof}

\begin{proposition}\label{svdp}
Smooth surfaces $X_{a,b}=\lbrace x^2y=z^2 - b + ax\rbrace$ have the algebraic volume density property with respect to $\omega=\mathrm{d}x/x^2\wedge\mathrm{d}z$.
\end{proposition}
\begin{proof}
Let $L$ be the set of function that corresponds to the Lie algebra of complete vector fields. By Lemma \ref{correspondance} we already have $x^i\in L$ and
$y^i\in L$ for $i\geq0$. We need to show that all functions on $X$ (modulo constants) are contained in $L$. It is
enough to show that (a) $x^iz^{k+1} \in L$, (b) $xy^{j+1}z^k\in L$ and (c) $y^{j+1}z^{k+1}\in L$ for all $i,j,k\geq 0$. 

\textit{First we show (a) $x^iz^{k+1}\in L$:} The statement (a) is also true for $k=-1$ by Lemma \ref{correspondance}. Lemma \ref{bracket} shows that
$v_y(x^{i+1})\in L$. Therefore by (\ref{2}) we get $2(i+1)x^iz\in L$ and thus $x^iz\in L$ for $i\geq 0$ which is the statement for $k=0$. Let us assume that
the statement is true for $k-1$ and for $k$. Then, by Lemma \ref{bracket} we have $v_y(x^{i+1}z^k)\in L$. By the induction assumption and (\ref{3}) we have also
$x^{i}z^{k+1}\in L$ which concludes the proof of (a) $x^iz^{k+1}\in L$ inductively for all $i,k\geq 0$.

\textit{The next step is to show (b) $xy^{j+1}z^k\in L$ and (c) $y^{j+1} z^{k+1}$ for $k=0$:} Note that (c) holds also for $k=-1$ by Lemma
\ref{correspondance}. By Lemma \ref{bracket} we have $v_x(y^{j+2})\in L$, and thus by (\ref{4}) $y^{j+1}z\in L$ which proofs statement (c) for $j\geq 0$ and
$k=0$. By the same lemma we have $y^jv_y(z)\in L$. Thus by (\ref{1}) and (c) we have $xy^{j+1}\in L$ proving statement (b) for $j\geq0$ and $k=0$.

\textit{The last step is to show (b) $xy^{j+1}z^k\in L$ and (c) $y^{j+1} z^{k+1}$ for arbitrary $k$:} Let us assume that (b) and (c) hold for $k$ and
moreover (c) holds for $k-1$. By Lemma \ref{bracket} and the induction assumption we have $v_x(y^{j+1}z^{k+1})\in L$ for all $j\geq0$. Thus by the induction
assumption and (\ref{5}) we also have $y^jz^{k+2}\in L$ for all $j\geq0$ which is statement (c) for $k+1$. Similarly, we have $y^jv_y(z^{k+2})\in L$, and thus 
by (\ref{1}) and the induction assumption we get $xy^{j+1}z^{k+1}\in L$ for all $j\geq0$. This is statement (b) for $k+1$. Thus by induction over $k$ the
statements (b) and (c) are shown.

\end{proof}
\begin{theorem}
 Let $X=\lbrace x^2y=p(x,z)\rbrace$ with $p\in\C[x,z]$ and $\deg(p(0,z))\leq 0$ be a smooth surface, then $X$ has the algebraic volume density property for the
volume form $\d x / x^2 \wedge dz$.
\end{theorem}
\begin{proof}
 If $\deg(p(0,z))\in\lbrace 1,2 \rbrace$ then by Proposition \ref{standard} the surface $X$ is algebraically isomorphic to some surface $X_{a,b}$ or to $\C^2$.
Since on these surfaces there are no non-constant invertible regular functions two different algebraic volume forms differ only by multiplication with a
constant. So the isomorphism induces a natural bijection between algebraic volume preserving vector fields on $X$ and $X_{a,b}$ (resp. $\C^2$). Thus algebraic
volume density property is preserved under algebraic isomorphisms. Therefore Proposition \ref{svdp} and the well known fact that $\C^2$ has the algebraic volume
density property concludes this case.

If $\deg(p(0,z))=0$ then by Proposition \ref{standard} the surface $X$ is algebraically isomorphic to $\C^*\times \C$. For any algebraic volume form $\eta$ on
$\C^*\times \C$ there is an algebraic automorphism such that the pull back of $\eta$ is equal to $\eta_0 = \d u/u\wedge \d v$. Indeed for an arbitrary $\eta =
au^k \d u \wedge \d v$ with $a\in\C^*$ and $k\in\mathbb{Z}$ the pull back of $\eta$ by $(u,v)\mapsto (u,a^{-1}u^{-k-1}v)$ is $\eta_0$. Apply Theorem 1 of
\cite{kaku-volume2} to the semi-compatible pair $(u\cdot\partial/\partial u,\partial/\partial v)$ to see that  $\C^*\times\C$ with $\eta_0$ has the algebraic
volume density property and thus $\C^*\times \C$ has the algebraic volume density
property for all algebraic volume forms.
\end{proof}

\subsection{The density property}
We will show the density property for the surfaces $X_{1,b}=\lbrace x^2y = z^2-b+x\rbrace$ with $b\in\C$. We will use the following vector fields on $X_{1,b}$:
\[
  v_x=2z\frac{\partial}{\partial y} + x^2\frac{\partial}{\partial z}, \quad  v_y=2z\frac{\partial}{\partial x} + (2xy-1)\frac{\partial}{\partial z}, \]\[ 
v_z=z^2x\frac{\partial}{\partial x} - \left(2(z^2-b)y -axy + a^2\right)\frac{\partial}{\partial y}.
\]
\begin{definition}
 For an algebraic vector field $\nu$ the $\omega$-divergence $\div \nu$ is the regular function given by $\d i_\nu\omega = (\div\nu)\cdot \omega$.
\end{definition}

\begin{lemma}\label{div'}
 We have $x^kv_x,y^kv_y\in\cavfw(X_{1,b})$ and $zx^kv_x,z^kv_z\in\cavf(X_{1,b})$ for $k\geq0$. Moreover $\div zx^kv_x = x^{k+2}$ and $\div z^kv_z = - z^{k+2}$.
\end{lemma}
\begin{proof}
 Lemma \ref{complete} says $x^kv_x,y^kv_y\in\cavfw(X_0)$ and $zx^kv_x\in\cavf(X_{1,b})$. For the statement about the divergence we make the calculations $$\d
i_{zx^kv_x}\omega = -\d zx^k\d x =  x^k\d x\wedge\d
z = x^{k+2}\omega,$$  $$\d i_{z^kv_z}\omega = \d \frac{z^{k+2}}{x}\d z = \frac{-z^{k+2}}{x^2}\d x \wedge \d z = -z^{k+2}\omega,$$ and thus prove the statement.
\end{proof}

The following lemma is well known.
\begin{lemma}\label{div}
For $v_1,v_2\in\avf(X_{1,b})$ we have
\[
 \div [v_1,v_2] =v_1(\div v_2) - v_2(\div v_1),
\]
in particular we have $$\div [x^kv_x,v_2] = x^kv_x(\div v_2) \quad \mathrm{and} \quad \div [y^kv_y,v_2] = y^kv_y(\div v_2)$$ for any $k\geq0$. Moreover for
$f\in\C[X_{1,b}]$:
\[
\div fv = f\div v + v(f).
\]
\end{lemma}

\begin{lemma}\label{lem-v_ysubmod}
 Let $f\in\C[X_{1,b}]$. Then $fv_y\in\lie(\cavf(X_{1,b}))$.
\end{lemma}
\begin{proof}
Let $E\subset\C[X_{1,b}]$ be given by $E=\lbrace \div\nu: \nu\in\lie(\cavf(X_{1,b}))\rbrace$. It is enough to show that for every $f\in\C[X_{1,b}]$ we have
$\div(fv_y)\in E$. Indeed, then $fv_y - \nu \in\avfw(X_{1,b})$ for some $\nu\in\lie(\cavf(X_{1,b}))$. Thus by Proposition \ref{svdp} we have
$fv_y\in\lie(\cavf(X_{1,b}))$. 

By Lemma \ref{div'} we have $x^{i+2}\in E$ and $z^{i+2}\in E$ for all $i\geq 0$. Thus by Lemma \ref{div} we get $v_y(x^{i+2})=2(i+2)x^{i+1}z\in E$, and
therefore $v_y(xz)=2z^2 + 2x^2y -x = 4z^2 - 2b + x \in E$. Since $z^2\in E$ we get $x-2b\in E$, and thus $v_y(x-2b)=2z\in E$. Altogether we have $x^iz\in E$,
$x^{i+2}\in E$ and $x-2b\in E$ for all $i\geq 0$.

Let $f=x^iy^jz^k$ for $i,j\geq 0$ and $k\in\lbrace 0,1\rbrace$. If $f\neq xy^j$ then we have
$$ \div (fv_y) = y^jv_y(x^iz^k) \in E$$
by Lemma \ref{div} because $x^iz^k\in E$ by the above. If $f=xy^j$ then
$$ \div (fv_y) = y^jv_y(x) = y^jv_y(x-2b) \in E$$
by the same arguments. Thus the lemma is proven since any regular function is a sum of such monomials by Remark \ref{rem-standard}.
\end{proof}

\begin{proposition}\label{prop-surfdp}
 The surface $X_{1,b}$ has the density property.
\end{proposition}
\begin{proof}
 By Lemma \ref{lem-genset} the tangent vectors of $v_y$ are a generating set for the tangent space $\T_qX_{1,b}$ at a generic point $q\in X_{1,b}$. By Lemma
\ref{lem-v_ysubmod} the $\C[X_{1,b}]$- submodule generated by $v_y$ is contained in $\lie(\cavf(X_{1,b}))$. Thus the $\O_{X_{1,b}}(X_{1,b})$- submodule
generated by $v_y$ is contained in the closure of $\lie(\cavf(X_{1,b}))$. Therefore by Lemma \ref{lem-genset} and Proposition \ref{prop-crit} the surface
$X_{1,b}$ has the density
property.
\end{proof}

\begin{theorem}
 Let $X=\lbrace x^2y=p(x,z)\rbrace$ with $p\in\C[x,z]$ and $\deg(p(0,z))\leq 2$ be a smooth surface. Then $X$ has the density property.
\end{theorem}
\begin{proof}
 By Proposition \ref{standard} the surface $X$ is biholomorphic to $\C^*\times \C$, $\C^2$, $X_{1,0}$ or $X_{1,1}$. It is well known that the first two have
the density property. The two other surfaces have the density property by Proposition \ref{prop-surfdp}.
\end{proof}

\end{document}